\documentclass[final,3p,times]{article}

\usepackage{lineno,hyperref}
\modulolinenumbers[1]

\usepackage{graphicx}
\usepackage{amssymb}
\usepackage{mathtools}
\usepackage{amsthm}
\usepackage{amsmath}
\usepackage{dsfont}
\usepackage{epsfig}
\usepackage{float}
\usepackage{epstopdf}
\usepackage{subfigure}
\usepackage{cite}
\usepackage{xcolor}
\usepackage{latexsym,amssymb}
\usepackage{pgfplots}
\usepackage{amsmath}
\usepackage{authblk}

\numberwithin{equation}{section}

\theoremstyle{definition}

\newtheorem{theorem}{Theorem}[section]
\newtheorem{lemma}[theorem]{Lemma}

\newtheorem{corollary}[theorem]{Corollary}
\newtheorem{definition}[theorem]{Definition}
\newtheorem{remark}[theorem]{Remark}

\theoremstyle{remark}

\usetikzlibrary{plotmarks}
\usetikzlibrary{external}
\pgfplotsset{compat=1.3}
\newlength\figurewidth 
\newlength\figureheight
\tikzset{external/force remake=true}

\newmuskip\pFqmuskip

\newcommand*\pFq[6][8]{%
	\begingroup 
	\pFqmuskip=#1mu\relax
	\mathchardef\normalcomma=\mathcode`,
	\mathcode`\,=\string"8000
	\begingroup\lccode`\~=`\,
	\lowercase{\endgroup\let~}\pFqcomma
	{}_{#2}F_{#3}{\left(\genfrac..{0pt}{}{#4}{#5}\bigg|#6\right)}%
	\endgroup
}
\newcommand{\pFqcomma}{{\normalcomma}\mskip\pFqmuskip}

\usepackage[margin=0.87in]{geometry}

\date{}

\begin{document}
	
	
	
	\title{Terminating Zero-Balanced Hypergeometric Series Using Divided Differences}

	\author[a]{Fatma Z\"{u}rnac{\i}-Yeti\c{s}  \footnote{\textbf{Email addresses:} $^a$fzurnaci@itu.edu.tr }}

	\affil[a]{Department of Mathematics Engineering, Istanbul Technical University,  Maslak, Istanbul, 34469, Turkiye}

	\setcounter{Maxaffil}{0}
	\renewcommand\Affilfont{\small}

	\maketitle
	
	\begin{abstract}
	There is a close relationship between divided differences and hypergeometric series, and recent studies have shown that divided differences can be used effectively to derive terminating hypergeometric identities.  In this paper, we apply this 
	approach to terminating zero-balanced hypergeometric series. Using 
	explicit product evaluations arising from divided differences, we give
	 alternative proofs of  zero-balanced ${}_3F_2(1)$ and 
	${}_4F_3(1)$ summation formulas and then  extend the argument to the general terminating ${}_{r+1}F_r(1)$ case. We also combine the Lagrange representation with the Leibniz rule 
	for divided differences to derive a finite convolution transformation 
	for a terminating ${}_{r+2}F_{r+1}(1)$ series. Thus, the divided-difference approach not only reproduces the known zero-balanced summation formulas but also yields a finite convolution transformation for terminating hypergeometric series, together with its zero-balanced specialization.

	\end{abstract}
	\section{Introduction}\label{sec1}
	
	There is a well-known connection between divided differences and 
	hypergeometric series. Divided differences play a fundamental role in 
	numerical analysis and approximation theory and are closely related to 
	Newton interpolation and B-spline approximation \cite{goldman}. 
	Hypergeometric series, on the other hand, arise in many areas of 
	mathematics and mathematical physics, including number theory, 
	combinatorics, special functions, and orthogonal polynomials
	~\cite{andrew,gasper,bailey}. Although these two subjects arise in different mathematical contexts, recent studies have shown that divided differences can be used effectively in the study of hypergeometric and basic hypergeometric identities.

	In the past decade, several authors have explored hypergeometric series identities from different perspectives \cite{chu1,chu2,fatma,fatma2,tuncer}.  In particular, Chu developed an approach based on divided differences for terminating well-poised hypergeometric and basic hypergeometric series 
	and obtained several classical summation and transformation formulas in this way
	~\cite{chu1,chu2}. These results suggest that divided differences are useful not only as 
	an interpolation device but also as a powerful tool for deriving terminating hypergeometric summation and transformation formulas.

	The present paper applies this point of view to terminating 
	zero-balanced hypergeometric series. In~\cite{zero}, several 
	zero-balanced terminating hypergeometric identities were considered, 
	including a general family whose validity was supported by numerical 
	computations for selected parameter values. Here we give direct divided-difference proofs of the relevant identities. We first obtain 
	divided-difference representations for terminating ${}_3F_2(1)$ and 
	${}_4F_3(1)$ series and use them to give alternative proofs of the 
	corresponding zero-balanced summation formulas. The same construction 
	is then extended to a general terminating ${}_{r+1}F_r(1)$ series, 
	yielding the general zero-balanced summation formula by a direct 
	degree argument.
	
	In addition to these summation results, we use the divided-difference 
	approach to derive a finite convolution transformation. More precisely, 
	a suitable product is evaluated first through the Lagrange representation 
	of divided differences and then through the Leibniz rule. Equating the 
	two representations produces a transformation from a terminating 
	${}_{r+2}F_{r+1}(1)$ series to a finite sum of terminating 
	${}_{r+1}F_r(1)$ series. We then show that this transformation has a zero-balanced specialization.
Thus, the same divided-difference approach provides summation formulas for terminating zero-balanced hypergeometric series and a finite transformation formula together with its zero-balanced specialization.

	This paper is organized as follows. In Section \ref{sec2}, some basic concepts are briefly reviewed. In Section \ref{subsec1}, the standard definition and properties of divided differences are given. In Section   \ref{subsec2}, the notation and a few simple identities for shifted factorials are established. In Section   \ref{subsec3},  the definition of generalized hypergeometric series and the zero-balance condition are recalled. In Section \ref{sec3}, we first derive divided-difference representations leading to the 
	zero-balanced ${}_3F_2(1)$ and ${}_4F_3(1)$ summations, and then extend the construction to the general terminating ${}_{r+1}F_r(1)$ case. Finally, a finite convolution transformation for 
	${}_{r+2}F_{r+1}(1)$ is obtained, together with its zero-balanced specialization.
	
	\section{Preliminaries}\label{sec2}
	
	\subsection{Divided differences}\label{subsec1}
	The standard definitions and properties of the divided differences will be given throughout this section.
	\begin{definition} Let     $f$  be an arbitrary function and $a< x_{0}\leq \ldots \leq x_{n} <b$   be a set of $n+1$  nodes. The divided difference of $f$  at these nodes is defined recursively as follows 
		\begin{align*}
			f[x_{0}]=&f(x_{0})\\
			f[x_{0},x_{1}]=&\frac{f(x_{1}) - f(x_{0})}{x_{1}-x_{0}} \quad  \quad &x_{1} \neq x_{0}\\=&f'(x_{0}) \quad  \quad \quad  \quad \quad   &x_{1} = x_{0}\\
			&\vdots \quad  \quad \quad  \quad \quad   \quad &\vdots \quad \quad   \\
			f[x_{0},\ldots,x_{n}]=&\frac{f[x_{1},\ldots,x_{n}] - f[x_{0},\ldots,x_{n-1}]}{x_{n}-x_{0}} \quad  \quad &x_{n} \neq x_{0}\\=& \frac{f^{(n)}(x_{0})}{n!}\quad  \quad \quad  \quad \quad \,\, \,\,  & x_{n} = x_{0}.
		\end{align*}
	\end{definition}
	There are numerous properties and identities of divided differences. We list some important properties of the divided difference below:\\
	\textit{Linearity:}
	\begin{align*}
		&(cf+g)[x_{0},x_{1},\ldots,x_{n}]= c(f[x_{0},x_{1},\ldots,x_{n}])+g[x_{0},x_{1},\ldots,x_{n}],\, c\in \mathbb{R}.
	\end{align*} 
	\textit{Lagrange coefficients:} For distinct nodes ${{x}_{0}},{{x}_{1}},\ldots ,{{x}_{n}},$  
	\begin{equation}\label{prop2}
		f[x_{0},x_{1},\ldots,x_{n}]=\sum_{k=0}^{n}\frac{f(x_{k})}{\prod_{  \substack{ j=0 \\ j\neq k}     }^{n} (x_{k}-x_{j})}. 
	\end{equation}
	\textit{Divided differences of polynomials: }Let $p(x)$ be a monic polynomial such that $\deg (p(x))=m$ and $m\le n,$ then 
	$$p[{{x}_{0}},{{x}_{1}},\ldots ,{{x}_{n}}]=\delta_{m,n},$$
	where $\delta_{m,n} $  is the Kronecker delta defined by $\delta_{m,n}= \left\{
	\begin{array}{ll}
		1, & m=n,  \\
		0, & m\neq n \\
	\end{array} 
	\right.$.\\
	The divided difference notation $[{{x}_{0}},{{x}_{1}},\ldots ,{{x}_{n}}]f$ is used instead of $f[{{x}_{0}},{{x}_{1}},\ldots ,{{x}_{n}}]$ for the sake of clarity of the operations.

	\subsection{Shifted factorials}\label{subsec2}
	Throughout this section, we shall adopt the following standard definitions and notation for the shifted factorials and the multiple shifted factorials.\\
	\textit{Shifted Factorials}
	\begin{align*}
		(\alpha)_{0}=1, \quad(\alpha)_{n}=\prod_{i=0}^{n-1}(\alpha+i) \quad n=1,2,\ldots.
	\end{align*}
	\textit{ Multiple Shifted Factorials}
	\begin{align*}
		(\alpha_{1},\alpha_{2},\ldots,\alpha_{m})_{n}=(\alpha_{1})_{n}(\alpha_{2})_{n}\cdots(\alpha_{m})_{n}.
	\end{align*}
	The following standard identities for shifted factorials will be used throughout the paper. For $\alpha\in\mathbb{C}$ and appropriate nonnegative integers $p,q,N$, we have
		\begin{align} \label{shiffac1}
		&(\alpha-N)_N=(-1)^N(1-\alpha)_N, \\
		& (\alpha)_{p+q}=(\alpha)_p(\alpha+p)_q \label{shiffac2}\\
		&	(\alpha)_{N-j}
		=
		\frac{(-1)^j(\alpha)_N}
		{(-\alpha-N+1)_j},
		\qquad 0\leq j\leq N,  \label{shiffac3}\\       &	(-N)_j=(-1)^j\frac{N!}{(N-j)!},
		\qquad 0\leq j\leq N. \label{shiffac4}
	\end{align}
	\subsection{Hypergeometric series} \label{subsec3}
	The $_{r}F_{s}$ hypergeometric series is defined by:
	\begin{align}\label{hyper0}
		\pFq[3.5]{r}{s}{a_{1},\ldots, a_{r}}{b_{1},\ldots, b_{s}}{z} = \sum\limits_{k=0}^{\infty} \frac{(a_{1},\ldots, a_{r})_{k}}{(b_{1},\ldots, b_{s})_{k}}\frac{z^{k}}{k!}.
	\end{align}
	$_{r}F_{s}$ is called a zero-balanced hypergeometric series in case $$\sum_{j=1}^{s}b_{j}-\sum_{i=1}^{r}a_{i}=0.$$
	If $a_{j}=-n$ for some  $n\in \mathbb{N}$, then $(a_{j})_{k}=(-n)_{k}=0$ for all $k \geq n+1$. Therefore, in this case, the right-hand side of (\ref{hyper0}) reduces to a finite sum:
	\begin{align}\label{hyper}
		\pFq[3.5]{r}{s}{-n, a_{2},\ldots, a_{r}}{b_{1}, b_{2},\ldots , b_{s}}{z} = \sum\limits_{k=0}^{n} \frac{(-n, a_{2},\ldots, a_{r})_{k}}{(b_{1},\ldots, b_{s})_{k}}\frac{z^{k}}{k!}.
	\end{align}
	
	
	\section{Hypergeometric summation and transformation formulas}\label{sec3}
	Using the sequences defined by
	\begin{align*}
		{{X}_{k}}=a+k,\quad {{B}_{i}}=a-b+i,\quad {{D}_{j}}=a-d+j,\quad {{S}_{t}}=a-s+t,   
	\end{align*}
	where $k=0,1,\ldots ,m,$ $i=1,\ldots ,n$, $j=1,\ldots ,l$, and $t=1,\ldots ,\gamma$, the following products are calculated: 
	\begin{align}\label{prod1}
		\prod_{i=0\atop i\neq k}^{m}\frac{1}{X_{k}-X_{i}}&=\prod_{i=0\atop i\neq k}^{m}\frac{1}{k-i}=\frac{(-m)_k}{k!}\frac{(-1)^m}{m!}, \\\label{prod2}
		\prod_{i=1}^{n}{(X_{k}-B_{i})}&=\prod_{i=1}^{n}(b+k-i)=\frac{(-1)^n(1-b)_{n}(b)_{k}}{(b-n)_{k}}, \\\label{prod3}
		\prod_{j=1}^{l}{(X_{k}-D_{j})}&=\prod_{j=1}^{l}(d+k-j)=\frac{(-1)^l(1-d)_{l}(d)_{k}}{(d-l)_{k}},\\ \label{prod4}
		\prod_{t=1}^{\gamma}{(X_{k}-S_{t})}&=\prod_{t=1}^{\gamma}(s+k-t)=\frac{(-1)^\gamma(1-s)_{\gamma}(s)_{k}}{(s-\gamma)_{k}}.
	\end{align}
	These product evaluations form the basic algebraic ingredients of the 
	divided-difference approach used throughout this section. They allow 
	divided differences at the equally spaced nodes $X_0,\ldots,X_m$ to be 
	rewritten directly as terminating hypergeometric series. We begin with 
	the ${}_3F_2(1)$ case.
	\begin{lemma}\label{lemma1}
		For the sequences $X_{k}=a+k$, $B_{i}=a-b+i$ and $D_{j}=a-d+j$  where $k=0,\ldots,m$, $i=1,\ldots,n$, and $j=1,\ldots,l$, the following divided difference can be expressed in terms of a hypergeometric series: 
		\begin{align}\label{lemma3.1}
			&[X_0,X_1,\ldots,X_m]\prod_{i=1}^{n}{(y-B_{i})}\prod_{j=1}^{l}{(y-D_{j})}\\&=(-1)^{n+m+l}\frac{(1-b)_{n}(1-d)_{l}}{m!}\,\pFq[4]{3}{2}{-m,b,d}{b-n,d-l}{1}.
		\end{align}
	\end{lemma}
	
	\begin{proof} Using equation (\ref{prop2}), we have
		\begin{align*}
			&[X_0,X_1,\ldots,X_m]\prod_{i=1}^{n}{(y-B_{i})}\prod_{j=1}^{l}{(y-D_{j})}\\&=\sum_{k=0}^{m}\frac{\prod_{i=1}^{n}{(X_{k}-B_{i})}\prod_{j=1}^{l}{(X_{k}-D_{j})}}{\prod_{i\neq k}{(X_{k}-X_{i})}}
		\end{align*}
		
		The products in denominator and numerator are
		evaluated explicitly by (\ref{prod1})-(\ref{prod3}). Substituting these closed
		forms into the above sum and collecting the $k$-independent factors, we obtain
		\begin{align*}
			&[X_0,\ldots,X_m]
			\prod_{i=1}^{n}(y-B_i)\prod_{j=1}^{l}(y-D_j)\\&
			=
			\frac{(-1)^{n+m+l}(1-b)_n(1-d)_l}{m!}
			\sum_{k=0}^{m}
			\frac{(-m)_k(b)_k(d)_k}{k!(b-n)_k(d-l)_k}.
		\end{align*}
		Invoking the definition of hypergeometric series (\ref{hyper}) yields the result. 
	\end{proof}
	
	Lemma 3.1 gives the required hypergeometric representation of the 
	divided difference. When the degree of the polynomial in the divided 
	difference is matched with the order of the divided difference, the 
	representation immediately yields the corresponding zero-balanced 
	summation formula.
	\begin{theorem}[A zero-balanced hypergeometric $_{3}\uppercase{F}_{2}$ series \cite{zero,zero2}]
		\begin{align}\label{zerob}
			\pFq[4]{3}{2}{-m,b,d}{b-n,d-m+n}{1}=\frac{m!}{(1-b)_{n}(1-d)_{m-n}}.
		\end{align}  
		where $n,m\in {{\mathbb{N}}_{0}}$,  $m\ge n$,\, $b,d ,b-n,d-m+n\in \mathbb{C}\backslash \mathbb{Z}_{0}^{-}.$
	\end{theorem}

	\begin{proof}
		Consider equation (\ref{lemma3.1}). The polynomial
		\[
		P(y):=\prod_{i=1}^{n}(y-B_i)\prod_{j=1}^{l}(y-D_j)
		\]
		is monic of degree $n+l$. Therefore, when  $0\leq n+l\leq m$, by the divided-difference property
		stated as highest-order coefficient of the polynomial interpolant,
		its $m$th divided difference satisfies
		\[
		[X_0,X_1,\ldots,X_m]\,P(y)=\delta_{m,n+l}.
		\]
		Thus,	using Lemma \ref{lemma1},  the following closed formula is obtained:
		\begin{align}\label{Pfaff0}
			\pFq[4]{3}{2}{-m,b,d}{b-n,d-l}{1}=\delta_{m,n+l}\frac{m!}{(1-b)_{n}(1-d)_{l}}.
		\end{align}		
		Setting $l=m-n$ yields $\delta_{m,n+l}=\delta_{m,m}=1$. Then using (\ref{Pfaff0}), we obtain equation (\ref{zerob}).  
	\end{proof}
	The same construction extends naturally by introducing one further 
	product block. This leads to a terminating ${}_4F_3(1)$ representation 
	and, under the corresponding degree condition, to the zero-balanced 
	${}_4F_3(1)$ summation.
	\begin{lemma} \label{lemma2} For the sequences $X_{k}=a+k$, $B_{i}=a-b+i$, $D_{j}=a-d+j$ and $S_{t}=a-s+t$  where $k=0,\ldots,m$, $i=1,\ldots,n$, $j=1,\ldots,l$ and $t=1,\ldots,\gamma$, the following divided difference can be expressed in terms of a hypergeometric series: 
		\begin{align*}
			&[X_0,X_1,\ldots,X_m]\prod_{i=1}^{n}{(y-B_{i})}\prod_{j=1}^{l}{(y-D_{j})}\prod_{t=1}^{\gamma}{(y-S_{t})}\\&=(-1)^{n+m+l+\gamma}\frac{(1-b)_{n}(1-d)_{l}(1-s)_{\gamma}}{m!}\,\pFq[4]{4}{3}{-m,b,d, s}{b-n,d-l, s-\gamma}{1}.
		\end{align*}
	\end{lemma}
	
	\begin{proof}
		By the Lagrange representation of divided differences \eqref{prop2}, we have
		\[
		[X_0,X_1,\ldots,X_m]\,
		\prod_{i=1}^{n}(y-B_i)\prod_{j=1}^{l}(y-D_j)\prod_{t=1}^{\gamma}(y-S_t)
		=
		\sum_{k=0}^{m}
		\frac{
			\displaystyle
			\prod_{i=1}^{n}(X_k-B_i)\prod_{j=1}^{l}(X_k-D_j)\prod_{t=1}^{\gamma}(X_k-S_t)
		}{
			\displaystyle
			\prod_{\substack{i=0\\ i\neq k}}^{m}(X_k-X_i)
		}.
		\]
		For the choice $X_k=a+k$, the numerator and denominator products in the summand
		are evaluated explicitly by \eqref{prod1}--\eqref{prod4}. Substituting these closed
		forms into the above sum and collecting the $k$-independent factors yield
		\[
		[X_0,\ldots,X_m]\,
		\prod_{i=1}^{n}(y-B_i)\prod_{j=1}^{l}(y-D_j)\prod_{t=1}^{\gamma}(y-S_t)
		=
		\frac{(-1)^{n+m+l+\gamma}(1-b)_n(1-d)_l(1-s)_\gamma}{m!}
		\sum_{k=0}^{m}
		\frac{(-m)_k(b)_k(d)_k(s)_k}{k!(b-n)_k(d-l)_k(s-\gamma)_k}.
		\]
		By the definition of the terminating hypergeometric series \eqref{hyper}, the
		finite sum is precisely
		\[
		{}_4F_3\!\left(\!\begin{matrix}-m,\,b,\,d,\,s\\ b-n,\,d-l,\,s-\gamma\end{matrix}\!\middle|\,1\right),
		\]
		which completes the proof.
	\end{proof}
	
	\begin{theorem}[A zero-balanced hypergeometric $_{4}\uppercase{F}_{3}$ series \cite{zero,zero2}]
		\begin{align}\label{zerob3}
			\pFq[4]{4}{3}{-m,b,d, s}{b-n,d-l, s-\gamma}{1}=\frac{m!}{(1-b)_{n}(1-d)_{l}(1-s)_{\gamma}}.
		\end{align}  
		where $n,m,l\in {{\mathbb{N}}_{0}}$,\, $\gamma=m-n-l$,\,  $m\ge n+l$,\,\,$b,d, s,b-n,d-l, s-\gamma\in \mathbb{C}\backslash \mathbb{Z}_{0}^{-}.$
	\end{theorem}
	
	\begin{proof} From Lemma \ref{lemma2}, the following closed formula is obtained for  $0\leq n+l+\gamma\leq m$:
		\begin{align}
			\pFq[4]{4}{3}{-m,b,d, s}{b-n,d-l, s-\gamma}{1}=\delta_{m,n+l+\gamma}\,\frac{m!}{(1-b)_{n}(1-d)_{l}(1-s)_{\gamma}}.\label{zerob2}
		\end{align}
		Setting $\gamma=m-n-l$ and using (\ref{zerob2}), we arrive at formula (\ref{zerob3}). \end{proof} 
	The preceding ${}_3F_2(1)$ and ${}_4F_3(1)$ cases reveal the general 
	pattern of the construction. We now replace the individual parameter 
	blocks by $r$ families and derive the corresponding 
	${}_{r+1}F_r(1)$ representation.
	\begin{lemma}\label{lemma3}
		For the sequences  ${{X}_{k}}=a+k$,  $A_{{{i}_{1}}}^{1}=a-{{\beta }_{1}}+{{i}_{1}},$ $A_{{{i}_{2}}}^{2}=a-{{\beta }_{2}}+{{i}_{2}},\ldots ,$ $A_{{{i}_{r}}}^{r}=a-{{\beta }_{r}}+{{i}_{r}},$  where $k=0,1,\ldots ,m$ and ${{i}_{j}}=1,\ldots ,{{k}_{j}}$ for $j=1,\ldots ,r,$ the following divided difference can be expressed in terms of a hypergeometric series:
		
		\begin{align*}
			& [{{X}_{0}},{{X}_{1}},\ldots ,{{X}_{m}}]\prod\limits_{{{i}_{1}}=1}^{{{k}_{1}}}{(y-A_{{{i}_{1}}}^{1})}\prod\limits_{{{i}_{2}}=1}^{{{k}_{2}}}{(y-A_{{{i}_{2}}}^{2})}\ldots \prod\limits_{{{i}_{r}}=1}^{{{k}_{r}}}{(y-A_{{{i}_{r}}}^{r})} \\ 
			&={{(-1)}^{m+{{k}_{1}}+{{k}_{2}}+\ldots +{{k}_{r}}}}\frac{{{(1-{{\beta }_{1}})}_{{{k}_{1}}}}{{(1-{{\beta }_{2}})}_{{{k}_{2}}}}\ldots {{(1-{{\beta }_{r}})}_{{{k}_{r}}}}}{m!}\, \pFq[4]{r+1}{r}{-m,{\beta }_{1},{\beta }_{2},\ldots, {\beta }_{r}}{{\beta }_{1}-{k}_{1}, {\beta }_{2}-{k}_{2},\ldots,{\beta }_{r}-{k}_{r}}{1}.
		\end{align*}
	\end{lemma}
	
	\begin{proof}For any $j>0$, we get
		\begin{align}\label{prod66}
			\prod\limits_{{{i}_{j}}=1}^{{{k}_{j}}}(X_{k}-A_{i_{j}}^{j})=(-1)^{k_{j}}\frac{{{(1-{{\beta }_{j}})}_{{{k}_{j}}}}{{({{\beta }_{j}})}_{k}}}{(\beta _{j}-k_{j})_{k}}.
		\end{align}
		Here, we follow an approach similar to that used in the proof of Lemma \ref{lemma2}. Using property (\ref{prop2}) together with equation (\ref{prod1}) and products (\ref{prod66}) for $j=1, 2, \ldots, r$, we obtain 	
		
		\begin{align*}
			& [{{X}_{0}},{{X}_{1}},\ldots ,{{X}_{m}}]\prod\limits_{{{i}_{1}}=1}^{{{k}_{1}}}{(y-A_{{{i}_{1}}}^{1})}\prod\limits_{{{i}_{2}}=1}^{{{k}_{2}}}{(y-A_{{{i}_{2}}}^{2})}\ldots \prod\limits_{{{i}_{r}}=1}^{{{k}_{r}}}{(y-A_{{{i}_{r}}}^{r})} \\&=\sum_{k=0}^{m}\frac{\prod\limits_{{{i}_{1}}=1}^{{{k}_{1}}}{(X_{k}-A_{{{i}_{1}}}^{1})}\prod\limits_{{{i}_{2}}=1}^{{{k}_{2}}}{(X_{k}-A_{{{i}_{2}}}^{2})}\ldots \prod\limits_{{{i}_{r}}=1}^{{{k}_{r}}}{(X_{k}-A_{{{i}_{r}}}^{r})}}{\prod_{i\neq k}{(X_{k}-X_{i})}}\\
			& =\sum\limits_{k=0}^{m}{{{(-1)}^{{{k}_{1}}}}}\frac{{{(1-{{\beta }_{1}})}_{{{k}_{1}}}}{{({{\beta }_{1}})}_{k}}}{{{({{\beta }_{1}}-{{k}_{1}})}_{k}}}{{(-1)}^{{{k}_{2}}}}\frac{{{(1-{{\beta }_{2}})}_{{{k}_{2}}}}{{({{\beta }_{2}})}_{k}}}{{{({{\beta }_{2}}-{{k}_{2}})}_{k}}}\ldots {{(-1)}^{{{k}_{r}}}}\frac{{{(1-{{\beta }_{r}})}_{{{k}_{r}}}}{{({{\beta }_{r}})}_{k}}}{{{({{\beta }_{r}}-{{k}_{r}})}_{k}}}{{(-1)}^{m}}\frac{{{(-m)}_{k}}}{m!\,k!} \\ 
			& = {{(-1)}^{m+{{k}_{1}}+{{k}_{2}}+\ldots +{{k}_{r}}}}\frac{{{(1-{{\beta }_{1}})}_{{{k}_{1}}}}{{(1-{{\beta }_{2}})}_{{{k}_{2}}}}\ldots {{(1-{{\beta }_{r}})}_{{{k}_{r}}}}}{m!}\sum\limits_{k=0}^{m}{\frac{{{({{\beta }_{1}})}_{k}}{{({{\beta }_{2}})}_{k}}\ldots {{({{\beta }_{r}})}_{k}}}{{{({{\beta }_{1}}-{{k}_{1}})}_{k}}{{({{\beta }_{2}}-{{k}_{2}})}_{k}}\ldots {{({{\beta }_{r}}-{{k}_{r}})}_{k}}}}\frac{{{(-m)}_{k}}}{k!}.
		\end{align*}
		Invoking the definition of hypergeometric series (\ref{hyper}), we reach the result.			
	\end{proof}
	As in the lower-order cases, the general zero-balanced summation is 
	obtained by matching the total degree of the polynomial with the order 
	of the divided difference.
	\begin{theorem}[Generalization of the Zero-Balanced Series \cite{zero}]
		\begin{align}\nonumber
			&\pFq[4]{r+1}{r}{-m,{\beta }_{1},{\beta }_{2},\ldots, {\beta }_{r}}{{\beta }_{1}-{k}_{1}, {\beta }_{2}-{k}_{2},\ldots,{\beta }_{r}-\left(m-{{k}_{1}}-{{k}_{2}}-\ldots -{{k}_{r-1}}\right)}{1}\\& \label{gzero2}= \frac{m!} {{{(1-{{\beta }_{1}})}_{{{k}_{1}}}}{{(1-{{\beta }_{2}})}_{{{k}_{2}}}}\ldots {{(1-{{\beta }_{r}})}_{{m-{{k}_{1}}-{{k}_{2}}-\ldots -{{k}_{r-1}}}}}},
		\end{align}  
		where
		$m,{{k}_{1}},{{k}_{2}},\ldots ,{{k}_{r-1}}\in {{\mathbb{N}}_{0}}$,$m\ge {{k}_{1}}+{{k}_{2}}+\ldots +{{k}_{r-1}}$, the remaining numerator and denominator parameters being neither zero nor negative integers.
	\end{theorem}
	
	\begin{proof} By Lemma \ref{lemma3}, as a particular case, the following closed formula is obtained for $0\leq {{k}_{1}}+{{k}_{2}}+\ldots +{{k}_{r}}\leq m $
		\begin{align}\nonumber
			&\pFq[4]{r+1}{r}{-m,{\beta }_{1},{\beta}_{2},\ldots, {\beta }_{r}}{{\beta }_{1}-{k}_{1}, {\beta }_{2}-{k}_{2},\ldots,{\beta }_{r}-{k}_{r}}{1}\\& \label{gzero1}= \delta_{m,{{k}_{1}}+{{k}_{2}}+\ldots +{{k}_{r}}}\frac{m!} {{{(1-{{\beta }_{1}})}_{{{k}_{1}}}}{{(1-{{\beta }_{2}})}_{{{k}_{2}}}}\ldots {{(1-{{\beta }_{r}})}_{{{k}_{r}}}}},
		\end{align}
		Setting  ${{k}_{r}}=m-{{k}_{1}}-{{k}_{2}}-\ldots -{{k}_{r-1}}$ and using (\ref{gzero1}), we arrive at formula (\ref{gzero2}). \end{proof} 
	
	The results obtained so far are summation formulas arising from the 
	evaluation of polynomial divided differences. We now use the same 
	divided-difference framework in a different way. Instead of reducing 
	the divided difference through a degree argument, we evaluate a product 
	in two different forms: directly by the Lagrange representation and by 
	the Leibniz rule for divided differences. Comparing these two evaluations 
	leads to a finite convolution transformation for terminating 
	hypergeometric series. The transformation itself is not restricted to 
	the zero-balanced case; the zero-balanced specialization will be 
	identified subsequently.
	\begin{theorem}
		\label{thm:finite-convolution-transformation}
		Let	$
		r,l\in\mathbb{N},
		\,\,
		m,n_1,\ldots,n_r\in\mathbb{N}_0,
		$
		and let
		$
		b_1,\ldots,b_r,d\in\mathbb{C}.
		$
		Assume that
		\begin{equation}
			\label{eq:parameter-condition-d}
			(1-d)_{m+l}\neq 0
		\end{equation}
		and
		\begin{equation}
			\label{eq:parameter-condition-b}
			(b_q-n_q)_m\neq 0,
			\qquad q=1,\ldots,r.
		\end{equation}
		Then
		\begin{equation}
			\label{eq:generalized-transformation}
			\begin{aligned}
				{}_{r+2}F_{r+1}\!\left[
				\begin{matrix}
					-m,\;b_1,\ldots,b_r,\;1-d\\
					b_1-n_1,\ldots,b_r-n_r,\;1-d+l
				\end{matrix};1
				\right]&=
				\frac{(1-d)_l}{(1-d)_{m+l}}
				\sum_{j=0}^{m}
				\binom{m}{j}
				(l)_{m-j}(1-d)_j
				\\
				&\qquad\times
				{}_{r+1}F_r\!\left[
				\begin{matrix}
					-j,\;b_1,\ldots,b_r\\
					b_1-n_1,\ldots,b_r-n_r
				\end{matrix};1
				\right].
			\end{aligned}
		\end{equation}
	\end{theorem}
	
		\begin{proof}
			Let
			\[
			X_k=a+k, \qquad k=0,1,\ldots,m,
			\]
			and define
			\[
			P(y)
			=
			\prod_{q=1}^{r}\prod_{i=1}^{n_q}
			\bigl(y-a+b_q-i\bigr)
			\]
			and
			\[
			G(y)
			=
			\frac{1}
			{\displaystyle\prod_{t=1}^{l}
				\bigl(y-a-d+t\bigr)}.
			\]
			Set
			\[
			R(y)=P(y)G(y).
			\]
			We evaluate the divided difference
			\[
			[X_0,X_1,\ldots,X_m]R
			\]
			in two different ways.
			We first evaluate the divided difference directly by the Lagrange
			representation:
			\begin{align}\label{lagrangerep}
			[X_0,\ldots,X_m]R
			=
			\sum_{k=0}^{m}
			\frac{R(X_k)}
			{\displaystyle\prod_{\substack{s=0\\ s\neq k}}^{m}
				(X_k-X_s)}.
		\end{align}
			For each $q=1,\ldots,r$, the product formula \eqref{prod66}, with
			$\beta_j=b_q$ and $k_j=n_q$, gives
			\[
			\prod_{i=1}^{n_q}
			(X_k-a+b_q-i)
			=
			(-1)^{n_q}(1-b_q)_{n_q}
			\frac{(b_q)_k}{(b_q-n_q)_k}.
			\]
			Set
			\[
			N=n_1+\cdots+n_r.
			\]
			Multiplying over $q=1,\ldots,r$, we obtain
			\begin{equation}
				P(X_k)
				=
				(-1)^N
				\prod_{q=1}^{r}(1-b_q)_{n_q}
				\prod_{q=1}^{r}
				\frac{(b_q)_k}{(b_q-n_q)_k}.
				\label{eq:PXk}
			\end{equation}
			Next,
			\[
			G(X_k)
			=
			\frac{1}{(1-d+k)_l}.
			\]
			Using \eqref{shiffac2}, we obtain
			\begin{equation}
				G(X_k)
				=
				\frac{1}{(1-d)_l}
				\frac{(1-d)_k}{(1-d+l)_k}.
				\label{eq:GXk}
			\end{equation}
			Moreover, 	it remains to calculate the denominator in the Lagrange
			representation. By \eqref{prod1},
			\begin{equation}
				\frac{1}
				{\displaystyle\prod_{\substack{s=0\\s\neq k}}^{m}
					(X_k-X_s)}
				=
				\frac{(-1)^m}{m!}
				\frac{(-m)_k}{k!}.
				\label{eq:Lagrange-coeff}
			\end{equation}
			
			Substituting \eqref{eq:PXk}, \eqref{eq:GXk}, and
			\eqref{eq:Lagrange-coeff} into the Lagrange representation \eqref{lagrangerep} yields
			\[
			[X_0,\ldots,X_m]R
			=
			\frac{(-1)^{N+m}}
			{m!(1-d)_l}
			\prod_{q=1}^{r}(1-b_q)_{n_q}
			\sum_{k=0}^{m}
			\frac{
				(-m)_k(b_1)_k\cdots(b_r)_k(1-d)_k
			}{
				(b_1-n_1)_k\cdots(b_r-n_r)_k
				(1-d+l)_k
			}
			\frac{1}{k!}.
			\]
			Therefore,
			\begin{equation}
				\begin{aligned}
					[X_0,\ldots,X_m]R
					&=
					\frac{(-1)^{N+m}}
					{m!(1-d)_l}
					\prod_{q=1}^{r}(1-b_q)_{n_q}
					\\
					&\quad\times
					{}_{r+2}F_{r+1}
					\left[
					\begin{matrix}
						-m,b_1,\ldots,b_r,1-d\\
						b_1-n_1,\ldots,b_r-n_r,1-d+l
					\end{matrix}
					;1
					\right].
				\end{aligned}
				\label{eq:first-evaluation}
			\end{equation}
			We next evaluate the same divided difference using the Leibniz rule.
			Since $R=PG$, we have
			\begin{equation}
				\label{eq:leibniz-divided-differences}
				[X_0,\ldots,X_m](PG)
				=
				\sum_{j=0}^{m}
				[X_0,\ldots,X_j]P\,
				[X_j,\ldots,X_m]G.
			\end{equation}
			
			We first evaluate $[X_0,\ldots,X_j]P$. Applying Lemma~\ref{lemma3} with
			$m$ replaced by $j$, $\beta_q=b_q$, and $k_q=n_q$, we obtain
			\begin{equation}
				\label{eq:P-divided-difference}
				\begin{aligned}
					[X_0,\ldots,X_j]P
					&=
					\frac{(-1)^{N+j}}{j!}
					\prod_{q=1}^{r}(1-b_q)_{n_q}
					\\
					&\quad\times
					{}_{r+1}F_r
					\left[
					\begin{matrix}
						-j,b_1,\ldots,b_r\\
						b_1-n_1,\ldots,b_r-n_r
					\end{matrix}
					;1
					\right],
				\end{aligned}
			\end{equation}
			where
			\[
			N=n_1+\cdots+n_r.
			\]
			It remains to evaluate $[X_j,\ldots,X_m]G$. Set
			\[
			h=m-j.
			\]
			By the Lagrange representation,
			\[
			[X_j,\ldots,X_m]G
			=
			\sum_{u=0}^{h}
			\frac{G(X_{j+u})}
			{\displaystyle
				\prod_{\substack{v=0\\v\neq u}}^{h}
				(X_{j+u}-X_{j+v})}.
			\]
			Since $X_{j+u}=a+j+u$, we have
			\[
			G(X_{j+u})
			=
			\frac{1}{(1-d+j+u)_l}.
			\]
			Using \eqref{shiffac2}, this becomes
			\begin{equation}
				\label{eq:G-node-value}
				G(X_{j+u})
				=
				\frac{1}{(1-d+j)_l}
				\frac{(1-d+j)_u}
				{(1-d+j+l)_u}.
			\end{equation}
			Moreover, applying \eqref{prod1} to the $h+1$ equally spaced nodes
			$X_j,\ldots,X_m$, we obtain
			\begin{equation}
				\label{eq:shifted-lagrange-coefficient}
				\frac{1}
				{\displaystyle
					\prod_{\substack{v=0\\v\neq u}}^{h}
					(X_{j+u}-X_{j+v})}
				=
				\frac{(-1)^h}{h!}
				\frac{(-h)_u}{u!}.
			\end{equation}
			Substituting \eqref{eq:G-node-value} and
			\eqref{eq:shifted-lagrange-coefficient} gives
			\[
			[X_j,\ldots,X_m]G
			=
			\frac{(-1)^h}
			{h!(1-d+j)_l}
			\sum_{u=0}^{h}
			\frac{(-h)_u(1-d+j)_u}
			{(1-d+j+l)_u}
			\frac{1}{u!}.
			\]
			Hence
			\[
			[X_j,\ldots,X_m]G
			=
			\frac{(-1)^h}
			{h!(1-d+j)_l}
			{}_2F_1
			\left[
			\begin{matrix}
				-h,1-d+j\\
				1-d+j+l
			\end{matrix}
			;1
			\right].
			\]
			By the Chu--Vandermonde summation formula,
			\[
			{}_2F_1
			\left[
			\begin{matrix}
				-h,\alpha\\
				\gamma
			\end{matrix}
			;1
			\right]
			=
			\frac{(\gamma-\alpha)_h}{(\gamma)_h},
			\qquad h\in\mathbb N_0,
			\]
			and therefore
			\[
			{}_2F_1
			\left[
			\begin{matrix}
				-h,1-d+j\\
				1-d+j+l
			\end{matrix}
			;1
			\right]
			=
			\frac{(l)_h}{(1-d+j+l)_h}.
			\]
			Consequently,
			\[
			[X_j,\ldots,X_m]G
			=
			\frac{(-1)^h(l)_h}
			{h!(1-d+j)_l(1-d+j+l)_h}.
			\]
		By \eqref{shiffac2}, we obtain
			\begin{equation}
				\label{eq:G-divided-difference}
				[X_j,\ldots,X_m]G
				=
				\frac{(-1)^{m-j}(l)_{m-j}}
				{(m-j)!(1-d+j)_{m+l-j}}.
			\end{equation}
		Substituting \eqref{eq:P-divided-difference} and
			\eqref{eq:G-divided-difference} into
			\eqref{eq:leibniz-divided-differences}, and using
			\[
			(-1)^{N+j}(-1)^{m-j}
			=
			(-1)^{N+m},
			\]
			we obtain
			\begin{equation}
				\label{eq:second-evaluation-pre}
				\begin{aligned}
					[X_0,\ldots,X_m]R
					&=
					(-1)^{N+m}
					\prod_{q=1}^{r}(1-b_q)_{n_q}
					\\
					&\quad\times
					\sum_{j=0}^{m}
					\frac{(l)_{m-j}}
					{j!(m-j)!(1-d+j)_{m+l-j}}
					\\
					&\quad\times
					{}_{r+1}F_r
					\left[
					\begin{matrix}
						-j,b_1,\ldots,b_r\\
						b_1-n_1,\ldots,b_r-n_r
					\end{matrix}
					;1
					\right].
				\end{aligned}
			\end{equation}
			To simplify the $j$-dependent denominator, observe that
			\[
			(1-d)_j(1-d+j)_{m+l-j}
			=
			(1-d)_{m+l}.
			\]
			Hence
			\[
			\frac{1}{(1-d+j)_{m+l-j}}
			=
			\frac{(1-d)_j}{(1-d)_{m+l}}.
			\]
			Using also
			\[
			\frac{1}{j!(m-j)!}
			=
			\frac{1}{m!}\binom{m}{j},
			\]
			equation \eqref{eq:second-evaluation-pre} becomes
			\begin{equation}
				\label{eq:second-evaluation}
				\begin{aligned}
					[X_0,\ldots,X_m]R
					&=
					\frac{(-1)^{N+m}}
					{m!(1-d)_{m+l}}
					\prod_{q=1}^{r}(1-b_q)_{n_q}
					\\
					&\quad\times
					\sum_{j=0}^{m}
					\binom{m}{j}
					(l)_{m-j}(1-d)_j
					\\
					&\quad\times
					{}_{r+1}F_r
					\left[
					\begin{matrix}
						-j,b_1,\ldots,b_r\\
						b_1-n_1,\ldots,b_r-n_r
					\end{matrix}
					;1
					\right].
				\end{aligned}
			\end{equation}
			We now compare \eqref{eq:first-evaluation} and
			\eqref{eq:second-evaluation}. Initially assume, in addition, that
			\[
			\prod_{q=1}^{r}(1-b_q)_{n_q}\neq0.
			\]
			Cancelling the common nonzero factors gives
			\[
			\begin{aligned}
				&{}_{r+2}F_{r+1}
				\left[
				\begin{matrix}
					-m,b_1,\ldots,b_r,1-d\\
					b_1-n_1,\ldots,b_r-n_r,1-d+l
				\end{matrix}
				;1
				\right]
				\\
				&\qquad=
				\frac{(1-d)_l}{(1-d)_{m+l}}
				\sum_{j=0}^{m}
				\binom{m}{j}
				(l)_{m-j}(1-d)_j
				\\
				&\qquad\qquad\times
				{}_{r+1}F_r
				\left[
				\begin{matrix}
					-j,b_1,\ldots,b_r\\
					b_1-n_1,\ldots,b_r-n_r
				\end{matrix}
				;1
				\right],
			\end{aligned}
			\]
			which is precisely \eqref{eq:generalized-transformation}. The remaining admissible cases follow by continuity.
			
	\end{proof}
	Theorem 3.7 holds without imposing the zero-balance condition. To 
	connect this transformation with the zero-balanced series considered 
	in the preceding results, we next compute the parametric excess of the 
	hypergeometric series on the left-hand side of~(3.16).
	\begin{remark}
		\label{rem:zero-balanced-case}
		Let
		\[
		N=n_1+\cdots+n_r.
		\]
		The parametric excess of the terminating
		\({}_{r+2}F_{r+1}(1)\) series on the left-hand side of
		\eqref{eq:generalized-transformation} is
		\begin{align*}
			&\left[
			\sum_{q=1}^{r}(b_q-n_q)+(1-d+l)
			\right]
			-
			\left[
			-m+\sum_{q=1}^{r}b_q+(1-d)
			\right]
			\\
			&\qquad
			=
			m+l-\sum_{q=1}^{r}n_q
			=
			m+l-N.
		\end{align*}
		Consequently, the series is zero-balanced whenever
		\[
		m+l=n_1+\cdots+n_r.
		\]
		Under this condition,
		\eqref{eq:generalized-transformation} provides a finite convolution
		transformation for a terminating zero-balanced
		\({}_{r+2}F_{r+1}(1)\) series.
	\end{remark}
	\begin{corollary}
		\label{cor:zero-balanced-transformation}
		Under the assumptions of
		Theorem~\ref{thm:finite-convolution-transformation}, suppose additionally
		that
		\[
		m+l=n_1+\cdots+n_r.
		\]
		Then the terminating hypergeometric series on the left-hand side of
		\eqref{eq:generalized-transformation} is zero-balanced. Consequently,
		\eqref{eq:generalized-transformation} provides the finite transformation
		formula
		\begin{equation}
			\label{eq:zero-balanced-transformation}
			\begin{aligned}
				&{}_{r+2}F_{r+1}\!\left[
				\begin{matrix}
					-m,\;b_1,\ldots,b_r,\;1-d\\
					b_1-n_1,\ldots,b_r-n_r,\;1-d+l
				\end{matrix};1
				\right]
				\\[2mm]
				&\quad=
				\frac{(1-d)_l}{(1-d)_{m+l}}
				\sum_{j=0}^{m}
				\binom{m}{j}
				(l)_{m-j}(1-d)_j
				\\
				&\qquad\qquad\times
				{}_{r+1}F_r\!\left[
				\begin{matrix}
					-j,\;b_1,\ldots,b_r\\
					b_1-n_1,\ldots,b_r-n_r
				\end{matrix};1
				\right],
			\end{aligned}
		\end{equation}
		for a terminating zero-balanced
		\({}_{r+2}F_{r+1}(1)\) series.
	\end{corollary}
	
	\begin{proof}
		By Remark~3.8, the condition
		\[
		m+l=n_1+\cdots+n_r
		\]
		makes the parametric excess of the series in~(3.16) equal to zero. 
		Hence the series is zero-balanced, and the result follows immediately 
		from Theorem~3.7.
	\end{proof}

\end{document}